\documentclass[10t,a4paper]{amsart}

\usepackage[english]{babel}
\usepackage[T1]{fontenc}
\usepackage[utf8]{inputenc}

\usepackage{amscd,amssymb,amsopn,amsmath,amsthm,graphics,amsfonts,enumerate,verbatim,calc}
\usepackage{amssymb, tikz}
\usepackage{mathtools}

\usepackage{hyperref}
\usepackage[capitalise]{cleveref}

\usepackage{setspace}
\usepackage{mathpazo}
\usepackage{color}
\usepackage{latexsym}
\usepackage{amsthm,amsfonts,amssymb,mathrsfs}
\usepackage{rotating}
\usepackage[leqno]{amsmath}
\usepackage{xspace}
\usepackage[all]{xy}
\usepackage{longtable}
\headsep=5 mm \numberwithin{equation}{section}
\newtheorem{theorem}{Theorem}[section]
\newtheorem{lemma}[theorem]{Lemma}

\newtheorem{corollary}[theorem]{Corollary}

\newtheorem{problem}[theorem]{Problem}
\theoremstyle{definition}
\newtheorem{definition}[theorem]{Definition}
\theoremstyle{remark}

\newtheorem{fact}[theorem]{Fact}
\newtheorem{example}[theorem]{Example}
\newtheorem{observation}[theorem]{Observation}

\newtheorem{question}[theorem]{Question}
\newtheorem{conjecture}[theorem]{Conjecture}
\newtheorem{hypothesis}[theorem]{Hypothesis}

\newcommand{\Ass}{\operatorname{Ass}}

\newcommand{\grade}{\operatorname{grade}}

\newcommand{\Spec}{\operatorname{Spec}}

\newcommand{\rad}{\operatorname{rad}}
\newcommand{\Tr}{\operatorname{Tr}}

\newcommand{\Ht}{\operatorname{ht}}
\newcommand{\pd}{\operatorname{pd}}

\newcommand{\Syz}{\operatorname{Syz}}

\newcommand{\Gdim}{\operatorname{Gdim}}

\newcommand{\V}{\operatorname{Var}}

\newcommand{\id}{\operatorname{id}}
\newcommand{\Ext}{\operatorname{Ext}}

\newcommand{\Supp}{\operatorname{Supp}}

\newcommand{\Tor}{\operatorname{Tor}}

\newcommand{\Hom}{\operatorname{Hom}}

\newcommand{\Ann}{\operatorname{Ann}}
\newcommand{\Hann}{\operatorname{h-Ann}}

\newcommand{\depth}{\operatorname{depth}}

\newcommand{\lo}{\longrightarrow}
\newcommand{\fm}{\frak{m}}
\newcommand{\fp}{\frak{p}}

\begin{document}

\author[]{Mohsen Asgharzadeh}

\address{}
\email{mohsenasgharzadeh@gmail.com}

\title[ ]
{homology of conductor}


\begin{abstract} 
We  study some homological properties of the conductor ideal.
\end{abstract}

\maketitle
\tableofcontents

\section{Introduction}

 The main theme of this paper
is about of the ideal of definition of the normal locus: $$\V(\mathfrak{C}):=\{\fp\in\Spec(R):R_{\fp} \emph{ is not normal}\},$$if it is exists. Under some mild conditions, for example the analytically unramified condition, it is closed in Zariski topology. Namely, and up to radical, we are interested in  $\mathfrak{C}:=\mathfrak{C}_R:=\Ann_R(\frac{\overline{R}}{R})$ and we refer it the conductor ideal. There are some  forgotten conjectures in this area:
\begin{conjecture}\label{1.7}(See \cite[Page 51]{mat2}).
	Suppose $R$ has no nilpotent. Then  $\frac{\overline{R}}{R}$ is $h$-reduced.
\end{conjecture}

We record \cite{weak} to see some partial answers. The other one was asked by Ikeda \cite{i}:

\begin{problem}\label{1.7} Let $\underline{x}$ be a system of parameter. When is $\mathfrak{C}\nsubseteq \underline{x}R$? 
\end{problem} Also,  it stated as an open problem in the book of Huneke-Swanson when the ring is Cohen-Macaulay, one may deal with the middle (see Conjecture \ref{41}).
We checked Problem \ref{1.7} in a lot of isolated examples, some of them are presented in Section 4.
The best answer to this is due to Ikeda over any complete Gorenstein integral domain. We present a new proof of this:

\textbf{Theorem A.} Adopt one of the following assumptions:
\begin{enumerate}
\item[i)] Suppose $R$ is quasi-Gorenstein. 
\item[ii)]  Suppose $\depth(R)=1$. 
\item[iii)] Suppose $R$ is a two-dimensional Buchsbaum ring.
\item[iv)] Suppose $\fm=\mathfrak{C}+y R$ for some $y$. 
\end{enumerate}
Then $\mathfrak{C}\nsubseteq \underline{x}R$ where $\underline{x}$  is a system of parameter.
	
We mention that there are several non-trivial examples for which $\fm=\mathfrak{C}+y R$. Next, we deal with:

\begin{conjecture}
$\mathfrak{C}_R$ can be used as a test object for any homological properties of rings.
\end{conjecture}

\textbf{Theorem B.} The following assertions are valid.
\begin{enumerate}
\item[i)]  If $\pd_R(\mathfrak{C})<\infty$, then $R$ is normal.
\item[ii)] If $\id_R(\mathfrak{C})<\infty$, then $R$ is normal. In particular, $R$ is Gorenstein.	 
\item[iii)] Suppose $R$ is $(S_2)$. If $\Ext_R^i(\mathfrak{C},\mathfrak{C})=0$ for some $i>0$, then
	$\mathfrak{C}$ is free. In particular,	$R$ is normal. 
\item[iv)]  If $\mathfrak{C}^{i\otimes}$  is torsion-free for some $i>1$ and $R$ is $(S_2)$, then
$\mathfrak{C}$ is free. In particular,	$R$ is normal.
\item[v)] If $R$ is one-dimensional, and $\Ext_R^i(\mathfrak{C},R)=0$ for some $i>0$
then $R$ is Gorenstein.
 
\end{enumerate}
Among other things, part i) uses the Roberts' intersection theorem. It is not so adjective any use of Bass' conjecture in the second item. Also, part iii) is related to a conjecture of
Auslander-Reiten.
Recall that part iv)  is a special case of
a conjecture by Vasconcelos, which is widely open.  We think the use of $(S_2)$-condition   is superficial. In this regard, we replace it with the seminormailty or even with quasi-Buchsbaum condition. Also, we present the Gorenstein analogue of i), see Observation \ref{g}. Part v) extends a result of
Ulrich \cite{U} and Rego \cite[main result]{Re}, as they worked with $\Ext_R^1(\mathfrak{C},R)=0$. It may be nice to note
that Rego used the compactified
Jacobian of a singular curve.
 In addition, we present a higher dimensional analogue. Reflexivity of $\overline{R}$
 seems more subtle, as we connect to initial Bass' numbers.
 In this regard, let $R$ be a 2-dimensional Buchsbaum ring which is $(R_1)$. In Fact \ref{ref} we show  $\overline{R}$ is reflexive iff $\mu^1(\fm ,R)\leq1$.

Commuting annihilator of modules is a difficult
task, and this become more subtle if the module under consideration is involved in the integral closure, as it is difficult to compute $\overline{R}$ too.
In this regard, there are several (co-)homological approaches to approximates the  annihilator. Namely, Vaconcelos' homological annihilator  denoted by $\Hann(-)$, and Bridger-Roberts' cohomological annihilator  denoted by $\Ann^h(-)$. For the precise definitions, see Section 3. We note these 
  (co-)homological approaches are not necessarily equal to $\Ann_R(-)$.
 In Section 3 we show:

\textbf{Theorem C.} The following assertions are valid.
	\begin{enumerate}
		\item[i)]Let $R$ be  a 2-dimensional Gorenstein ring. Then $\Hann(\frac{\overline{R}}{R})=\Ann_R(\frac{\overline{R}}{R})$.	\item[ii)] Let $R$ be  Cohen-Macaulay. Then
	  $\mathfrak{C}_R\subseteq \Ann^h(\overline{R})$. If $R$ is not normal, then $\rad \mathfrak{C}_R \neq \rad \Ann^h(\overline{R})$.
			\item[iii)]
		Let $R$ be  as ii). Then $\mathfrak{C}_R\subseteq \Ann\Tor^R_+(\overline{R},R/\underline{x}R)=0$.
			 \item[iv)] Let $R$ be  a 2-dimensional analytically unramified ring satisfying $(R_1)$. Then $\mathfrak{C}_R=\Ann(H^1_\fm(R))$.\item[v)]If $R$ is Gorenstein and $\overline{R}$ is Cohen-Macaulay, then $\mathfrak{C}_R$ is maximal Cohen-Macaulay. 
	\end{enumerate}

Concerning iii), it is true in slightly more general setting than be Cohen-Macaulay. Also, it is trivial if $i>1$ or even $\overline{R}$ is Cohen-Macaulay, and  maybe there is a related result from literature. Namely, over complete 1-dimensional integral domains, and for any maximal Cohen-Macaulay module $M$,  Wang proved that 
$\mathfrak{C}_R\Tor^R_1(M,-)=0$, see \cite[3.1+1.5]{wa}. Concerning v), and when the ring is standard graded over a field, it is proved in  \cite{Ei}  independently by 
Eisenbud and Ulrich.

We follow the book of Huneke-Sawnson \cite{HS} for all unexplained notations. Also, see the books \cite{BH,mat} as classical references.

\section{$\mathfrak{C}$ as a test module}

Let $(R,\fm)$ be a commutative noetherian local ring, and let $\overline{R}$ be the normalization of $R$ in the total ring of fraction. Also, $\underline{x}$  is a system of parameter for $R$.\begin{hypothesis}
	We assume in addition that the ring is analytically unramified. \end{hypothesis}
Recall that
Serre's $(S_n)$ condition  indicates 
$\depth(R_{\fp})\geq\min\{n,\Ht(\fp)\}$.
Serre's $(R_n)$ condition
means $R_{\fp}$ is regular for all prime ideal $\fp$ of height at most $n$. 

\begin{fact}(Serre)\label{sc}. $R$ is normal iff it satisfies   $(S_2)$ and $(R_1)$.\end{fact}
 By $\pd_R(-)$ we mean the projective dimension  of $(-)$.
\begin{lemma}\label{op}Suppose $R$ is $(S_2)$.
	If $\pd_R(\mathfrak{C})<\infty$, then
	$R$ is normal. 
\end{lemma}

\begin{proof}
Let $\fp$ be in $\Spec(R)$ of height one.
Recall that $\mathfrak{C}=\Hom_R(\overline{R},R)$. Since $R$ is analytically unramified, $\overline{R}$ is finitely generated as an $R$-module. In particular,
$$\Hom_R(\overline{R},-)_\fp	=\Hom_{R_\fp}(\overline{R}_\fp,-_\fp)=\Hom_{R_\fp}(\overline{R_\fp},-_\fp).$$From this, $(\mathfrak{C}_R)_\fp=\mathfrak{C}_{R_\fp}$. So,   projective dimension
of $(\mathfrak{C}_R)_\fp$ is  finite. Recall that
$$\pd_{R_\fp}(\frac{R_\fp}{(\mathfrak{C}_R)_\fp})\leq  \depth(R_\fp)\leq \dim(R_\fp)=1,$$and consequently, $(\mathfrak{C}_R)_\fp=\mathfrak{C}_{R_\fp}$ is free.
In particular, it should be a principal ideal. Recall from \cite[Ex. 12.4]{HS} that $\mathfrak{C}_{R_\fp}$ is not subset of any proper principal ideal. In sum, $\mathfrak{C}_{R_\fp}={R_\fp}$.
 This implies
${R_\fp}$ is normal. By dimension consideration,
${R_\fp}$ is regular. Following Fact \ref{sc},
$R$ is normal.
\end{proof}

Let us drop the $(S_2)$ condition.
\begin{observation}\label{2p} If $\pd_R(\mathfrak{C})<\infty$, then
	$R$ is normal.
\end{observation}

\begin{proof}We processed by induction on $d:=\dim R$. Recall that $R$ is analytically unramified. By definition, its completion with respect to $\fm$-adic topology is reduced, see \cite{HS}. Zero-dimensional reduced rings are regular. This proves the claim for $d=0$. Now, suppose $d>0$ and assume inductively that $R$ is locally normal over punctured spectrum. Recall that both assumptions $\pd_R(\mathfrak{C})<\infty$ and  analytically unramified behaved well with respect to localization. It turns out that $\mathfrak{C}_{R_\fp}=R_\fp$ for all $\fp\in\Spec(R)\setminus\{\fm\}$, i.e., $\frac{{R}}{\mathfrak{C}}$ is of finite length. 	Suppose on the way of contradiction that $\frac{{R}}{\mathfrak{C}}\neq 0$. Thanks to Lemma \ref{op}
	$R$ is not $(S_2)$.
Since $\frac{{R}}{\mathfrak{C}}$ is nonzero, of finite length and
of finite projective-dimension, and in the light of Roberts' intersection theorem \cite[6.2.4]{int1}, we see $R$ is Cohen-Macaulay. In particular, $R$ is $(S_2)$.  This contradiction shows $\frac{{R}}{\mathfrak{C}}= 0$, so 	$R$ is normal.
\end{proof}

By $\id_R(-)$ we mean the injective dimension  of $(-)$.
\begin{observation}\label{oi}
If $\id_R(\mathfrak{C})<\infty$, then
		$R$ is normal and in particular, $R$ is Gorenstein.	
\end{observation}
\begin{proof} By Bass' conjecture, which is a theorem,  $R$ is Cohen-Macaulay (see \cite{PS} and  \cite{int1}). In particular, $R$ is $(S_2)$.
Let $\fp\in\Spec(R)$ be of height one. Recall
that $(\mathfrak{C}_R)_\fp=\mathfrak{C}_{R_\fp}$ and that 
	$\mathfrak{C}_{R_\fp}$ is integrally closed, as $\dim({R_\fp})=1$ and 	$\mathfrak{C}_{R_\fp}$ is a common ideal of ${R_\fp}$ and $\overline{R_\fp}$ (see \cite[Ex. 12.2]{HS}). Due to Serre's criteria
	we may assume ${R_\fp}\neq\overline{R_\fp}$. 
	This yields that $\mathfrak{C}_{R_\fp}$ is primary to the maximal ideal, and integrally closed. This in turn implies that $\mathfrak{C}_{R_\fp}$ is full. Also 
	 $\id_{R_\fp}(\mathfrak{C}_\fp)<\infty$. In view of \cite[Theorem 2.2]{goto} we observe that
	${R_\fp}$ is regular.
	Again, by Serre's criteria, $R$ is normal. Then $\mathfrak{C}=R$ has finite injective dimension. So, $R$ is Gorenstein.
\end{proof}

Let $I$ be an ideal generated by $f_{1},\ldots,f_{r}$ of $R$. For an $R$-module $M$, we have the \v{C}ech complex
\begin{equation}
{C}^{\bullet}(I;M) : 0 \lo M \lo \prod_{1 \leq j \leq r}M_{f_{j}} \lo \prod_{1 \leq j_{1} < j_{2} \leq r} M_{f_{j_{1}}f_{j_{2}}} \lo \cdots \lo M_{f_{1}\cdots f_{r}} \lo 0.
\end{equation}
Then its $i$-th cohomology is  denoted by $H^{i}_{I}(M)$.
\begin{observation}\label{ec}Suppose $\Ext_R^i(\mathfrak{C},\mathfrak{C})=0$ for some $i>0$ where $R$ is one of the following:
	\begin{enumerate}
		\item[i)]$R$ is $(S_2)$
	\item[ii)]$R$ is quasi-Buchsbaum\footnote{here, and elsewhere we only need $\fm H^1_{\fm}(R)=0$.}, or 	\item[iii)]$R$ is seminormal.	\end{enumerate}
Then
		$\mathfrak{C}$ is free and in particular	$R$ is normal. 
\end{observation}

\begin{proof}i) Suppose first that  $R$ is $(S_2)$ Let $\fp$ be in $\Spec(R)$ of height one. Let $A:={R_\fp}$. Recall
that $(\mathfrak{C}_R)_\fp=\mathfrak{C}_A$. 
We look at the following exact sequence taken from \cite{ab}$$
\Tor_2^A(\Tr(\Syz_i(\mathfrak{C}_A)),\mathfrak{C}_A)\rightarrow
\Ext^i_A(\mathfrak{C}_A,A)\otimes_A \mathfrak{C}_A\rightarrow\Ext^i_A(\mathfrak{C}_A,\mathfrak{C}_A)\rightarrow\Tor_1^A(\mathfrak{C}_A,\Tr(\Syz_i(\mathfrak{C}_A)))\rightarrow 0,$$
where
$\Syz_i(-)$ is the $i^{th}$-syzygy of a module  $(-)$ and
 $\Tr(\Syz_i(\mathfrak{C}_A))$
 stands for the Auslander's transpose.
  In order to define it, we first set $(-)^\ast:=\Hom_A(-,A)$ and take $A^n\to A^m\to \Syz_i(\mathfrak{C}_A) \to 0$ be 
 a finite type presentation. Now, by taking dual, $\Tr$ is defined by the following cokernel:
 $$ 0\lo \Syz_i(\mathfrak{C}_A)^\ast\lo (A^\ast)^m\lo (A^\ast)^n\lo\Tr(\Syz_i(\mathfrak{C}_A)) \lo 0.$$From this  $\Tor_1^A(\mathfrak{C}_A,\Tr(\Syz_i(\mathfrak{C}_A)))=0$. By shifting $$\Tor_2^A(A/\mathfrak{C}_A,\Tr(\Syz_i(\mathfrak{C}_A)))=0\quad(\ast)$$ Without loss
of generality $\mathfrak{C}_A$ is integrally closed and primary to maximal ideal of $A$. This along with $(\ast)$ implies that $\pd(\Tr(\Syz_i(\mathfrak{C}_A)))<2$ (see \cite[3.3]{corso}).
Since  $(\Syz_i(\mathfrak{C}_A))^\ast$ is the second syzygy of  $\Tr(\Syz_i(\mathfrak{C}_A))$ we deduce that
 $  \Syz_i(\mathfrak{C}_A)  ^\ast $ is free. Thanks to \cite[Lemma 3.9]{dao} we know that $\Syz_i(\mathfrak{C}_A)$ is free. Then $\pd(\mathfrak{C}_A)<\infty$. In view of Observation \ref{op} we deduce that $A$ is normal and via the dimension consideration, $A$ becomes regular. We proved that $R$ satisfies $(R_1)$. According to	Fact \ref{sc}, $R$ is normal. In particular, $\mathfrak{C}=R$ is free.
 
ii) Suppose now that $R$ is  quasi-Buchsbaum. 
 We processed by induction on $d:=\dim R$. Recall that $R$ is analytically unramified.  This proves the claim for $d=0$. Now, suppose $d>0$ and assume inductively that $R$ is locally normal over punctured spectrum.  It turns out that  $\frac{{R}}{\mathfrak{C}}$ (resp. $\frac{\overline{R}}{R}$) is of finite length. 	Suppose on the way of contradiction that $\frac{{R}}{\mathfrak{C}}\neq 0$. Thanks to the first case, and without loss of generality, we assume in addition that
 $R$ is not $(S_2)$. In fact, following part i) we may and do assume in addition that $d>1$, as 1-dimensional reduced rings are $(S_2)$. This allows us to deal with the annihilator of 
$H^1_\fm(R)$ as follows. Namely, from  $0\to R\to  \overline{R}\to \frac{\overline{R}}{R} \to 0$ we deduce $0=H^0_\fm(\overline{R})\to H^0_\fm( \frac{\overline{R}}{R})\to H^1_\fm(R)\to H^1_\fm(\overline{R})=0.$ From this $H^0_\fm( \frac{\overline{R}}{R})=\frac{\overline{R}}{R}$.
 Consequently, $H^1_\fm(R)=H^0_\fm( \frac{\overline{R}}{R})=\frac{\overline{R}}{R}$, so $\mathfrak{C}_R=\Ann(H^1_\fm(R))=\Ann(\oplus_X k)$ where $X$ is an index set. Since
 $\mathfrak{C}_R\neq R$, $X$ is non-empty. It follows that
 $\mathfrak{C}=\fm$.
 Since $\Ext^i_R(\fm,\fm)=0$ we deduce $\Ext^{i+1}_R(R/ \fm,\fm)=0$. In other words, $\mu^{i+1}(\fm)=0$. We know from  \cite[Theorem II]{Ro2} that $i+1\not\in[\depth{\fm}=1,\id_R(\fm)]$. Thus $\id_R(\fm)<\infty$, i.e., $R$ is regular a contradiction.
Thus,	$\frac{{R}}{\mathfrak{C}}=0$, and the claim follows.

iii) We processed by induction on $d:=\dim R$. Recall that $R$ is analytically unramified. Since zero-dimensional reduced rings are regular, we proved the claim for $d=0$. Now, suppose $d>0$ and assume inductively that $R$ is locally normal over punctured spectrum. Recall that all three assumptions, seminormality,  $\Ext_R^i(\mathfrak{C},\mathfrak{C})=0$ and  analytically unramified behaved well with respect to localization. It turns out that $\mathfrak{C}_{R_\fp}=R_\fp$ for all $\fp\in\Spec(R)\setminus\{\fm\}$, i.e., $\frac{{R}}{\mathfrak{C}}$ is of finite length.
	Suppose on the way of contradiction that $\frac{{R}}{\mathfrak{C}}\neq 0$.
	Then $\mathfrak{C}$ is primary to a maximal ideal of normalization. In particular, $R$ is not regular.
	By seminormality, we know $\mathfrak{C}$ is a radical ideal in $\overline{R}$. Thus,  $\mathfrak{C}=\fm_{\overline{R}}$. Then $\mathfrak{C}=\mathfrak{C}\cap R=\fm_{\overline{R}}\cap R=\fm$, see \cite[page 62]{mat}. Since $\Ext^i_R(\fm,\fm)=0$ we deduce a contradiction.
Thus,	$\frac{{R}}{\mathfrak{C}}=0$, and the claim follows.
\end{proof}

\begin{definition}
	\begin{enumerate}
		\item[i)]
		By  $(G_n)$ we mean  that $R_{\fp}$  is Gorenstein for any $\fp$ of height at most  $n$.	
		\item[ii)] Recall from \cite{wol} that a ring $R$ is called quasi-normal  if it is  $(S_2)$ and $(G_1)$.
	\end{enumerate}
\end{definition}

By $\Gdim_R(-)$ we mean the G-dimension of $(-)$.

\begin{observation}\label{g}Suppose $R$ is $(S_2)$ and $\Gdim_R(\mathfrak{C})<\infty$.
	Then
	$R$ is quasi-normal.
\end{observation}

\begin{proof}
	This is similar to Observation \ref{oi}, and we leave the straightforward modification to the reader.
\end{proof}

We say a module $M$ is reflexive if the natural map $\delta_M:M\to M^{\ast\ast}$ is an isomorphism. One may ask the reflexivity problem of the $R$-module $\overline{R}$. 
\begin{fact}\label{ref}
i) Suppose $\dim R=1$ then $\overline{R}$ is reflexive.

ii) Suppose  $R$ is $(S_2)$. Then  $\overline{R}$ is reflexive.

iii) Let $R$ be 2-dimensional Buchsbaum ring which is $(R_1)$. If $\overline{R}$ is reflexive, then $\mu^1(\fm ,R)\leq1$.

iv) Let $R$ be 2-dimensional Buchsbaum ring which is $(R_1)$ and $\mu^1(\fm ,R)\leq1$. Then $\overline{R}$ is reflexive. 
\end{fact}

\begin{proof}
i) Following \cite[Ex. 12.13]{HS} we know $\mathfrak{C}^\ast=\overline{R}$. But, $\mathfrak{C}=\Hom_R(\overline{R},R)$. Combining these $\overline{R}^{\ast\ast}=\mathfrak{C}^\ast=\overline{R},$ it follows the natural map
$\overline{R}\to \overline{R}^{\ast\ast}$ is an isomorphism, and so $\overline{R}$ is reflexive.\footnote{Alternatively, use the fact that any dual module
	over Cohen-Macaulay rings is reflexive.}

ii) Let $\fp\in \Spec(R)$ be so that  $\depth R_{\fp} \geq 2$.
Recall that the map $R_{\fp}\lo \overline{R}_{\fp}$ is finite, and in view of \cite[1.2.26(b)]{BH}
$\depth_{{R}_{\fp}}(\overline{R}_{\fp})=\depth_{\overline{R}_{\fp}}(\overline{R}_{\fp})\geq 2$. Recall from i) that the claim is clear in dimension one, so:
\begin{enumerate}[(a)]
	\item $\overline{R}_{\fp}$ is reflexive for all $\fp$ with  $\depth R_{\fp} \leq 1$, and
	\item $\depth (\overline{R})_{\fp} \geq 2$ for all $\fp$ with  $\depth R_{\fp} \geq 2$.
\end{enumerate}
From this, and in the light of \cite[Proposition 1.4.1(b)]{BH}, $\overline{R}$ is reflexive.

iii) Suppose first that $R$ is Cohen-Macaulay. Then $\depth(R)=2$ and  $\mu^1(\fm ,R)=0<1$. Now, suppose
$R$ is 2-dimensional Buchsbaum ring which is  $(R_1)$ and not Cohen-Macaulay.  
Recall that $H^1_\fm(R)=H^0_\fm( \frac{\overline{R}}{R})=\frac{\overline{R}}{R}$, so $\mathfrak{C}_R=\Ann(H^1_\fm(R))=\Ann(\oplus_n k)$ where $n$ is an integer. Since
$\mathfrak{C}_R\neq R$, $n>0$ and consequently, $\mathfrak{C}_R=\fm$.
In particular, there is an exact sequence $$0\lo R\lo \overline{R} \lo k^n\lo 0\quad(\ast)$$Recall that $\depth(R)=1$. This shows that $(\frac{{R}}{\fm})^\ast=0$ and  $\Ext_R^1(\frac{{R}}{\fm},R)=k^{t}$ where $t:=\mu^1(\fm ,R)\neq 0$.
By dual $(\ast)$, we have $$\begin{CD}
0@>>> \overline{R}^{\ast} @> >> {R}^{\ast} @>>> \Ext_R^1(k^n,R)=k^{nt}@>>>\Ext_R^1(R,R)=0\\
\end{CD}	$$
Another duality
yields that $$0\lo {R}^{\ast\ast}\lo \overline{R}^{\ast\ast} \lo  \Ext_R^1(k^{nt},R)=k^{nt^2}\lo \Ext_R^1(R^\ast,R)=0.$$
Recall by work of Eilenberg-Maclane that the  biduality is a functor. This means by applying it to $M\to N$ yields that  
$$\begin{CD}
 M^{\ast\ast}  @> >>  N^{\ast\ast}\\
\delta_{M}@AAA\delta_{N} @AAA    \\
M@>>>N.\\
\end{CD}	$$
There is a map $g$ completes the following depicted diagram:
$$\begin{CD}
0@>>> R^{\ast\ast} @> >> \overline{R}^{\ast\ast} @>>>k^{nt^2} @>>> 0\\
@.\delta_{{R}}@AAA \delta_{\overline{R}}@AAA \exists g @AAA    \\
0@>>>R @>>> \overline{R} @>>>k^{n}@>>>0.\\
\end{CD}	$$\iffalse	
Suppose on the way of contradiction that $\delta_{\overline{R}}:\overline{R}\to \overline{R}^{\ast\ast}$  is an isomorphism. 
Thanks to 5-lemma it turns out that  $\delta_{\overline{R}/R}:k\overline{R}\to k^{\ast\ast}$ is an isomorphism. But, $k^{\ast\ast}=0$ and we get to a contradiction. So,
  $\overline{R}$ is not reflexive, as claimed.\fi
 If $\delta_{\overline{R}}:\overline{R}\to \overline{R}^{\ast\ast}$  is an isomorphism,
 it shows $g$ is an isomorphism, and so $t=1$.
 
 iv) As the claim is clear when $R$ is Cohen-Macaulay, we may assume  $\mu^1(\fm ,R)=1$. Recall
 that
 there is the following diagram:$$\begin{CD}
 R^{\ast\ast} @> >> \overline{R}^{\ast\ast} @>>>k^n  @>>> 0@>>>0\\
  \delta_{{R}}@AAA \delta_{\overline{R}}@AAA g @AAA  h_4 @AAA h_5 @AAA  \\
R @>>> \overline{R} @>>>k^n @>>>0@>>>0.\\
 \end{CD}	$$
   Recall that  $h_2:= \delta_{\overline{R}}$ and $h_4$ are injective and $h_1:=\delta_{{R}}$ is surjective, then by a version of 5-lemma  $h_3:=g$ is injective. But, any injective endomorphism
   of a finite dimensional vector space is indeed an isomorphism. Thus, $g$ is an isomorphism. In view of $$\begin{CD}
   0@>>> R^{\ast\ast} @> >> \overline{R}^{\ast\ast} @>>>k^{n} @>>> 0\\
   @.\cong @AAA \delta_{\overline{R}}@AAA  g @AAA    \\
   0@>>>R @>>> \overline{R} @>>>k^{n}@>>>0\\
   \end{CD}	$$we deduce $\delta_{\overline{R}}$ is an isomorphism. So, the claim follows.
\end{proof}

\begin{corollary} The ring $R$ is quasi-normal, provided all of the following three assertions hold:
\begin{enumerate}
\item[i)]$\dim R=2$
\item[ii)] $\Gdim_R(\mathfrak{C})<\infty$, and 
\item[iii)]$\overline{R}$ is reflexive.
	\end{enumerate}
\end{corollary}

\begin{proof}
According to the argument of  Observation \ref{g}, $R$ satisfies $(G_1)$. Suppose on the way of contradiction  that the ring is not $(S_2)$. So, $\depth(R)=1$. Thanks
to Auslader-Bridger formula, $\mathfrak{C}$ is totally reflexive. It turns out that
its dual is also totally reflexive. But, $\mathfrak{C}^\ast=\overline{R}^{\ast\ast}\stackrel{iii)}=\overline{R}$. Again, duo
to Auslader-Bridger formula we have 
$$1=\depth_R(R)=\Gdim_R(\overline{R})+\depth_R(\overline{R})=
\Gdim_R(\overline{R})+\depth_{\overline{R}}(\overline{R})=2,$$ a contradiction.
Thus the ring satisfies $(S_2)$. By  Observation \ref{g}, $R$ is quasi-normal.
\end{proof}

Here, we recall that $M^{(n+1)\otimes }:=M^{n\otimes }\otimes_RM$. Vasconcelos \cite{2010} asked:

\begin{question}
	Let $R$ be a local domain and $M$ be  torsion-free. Is there an integer $e$ guaranteeing that if $M$ is not free,
	then the tensor power $M^{e\otimes }$ has nontrivial torsion?
\end{question}See \cite{finitsup} for some partial answers. Here, we prove it for conductor ideal:
\begin{observation}Suppose   $\mathfrak{C}^{i\otimes}$  is torsion-free for some $i>1$,
	 where $R$ is one of the following:
	\begin{enumerate}
		\item[i)]$R$ is $(S_2)$
		\item[ii)] $R$ is quasi-Buchsbaum, or 	\item[iii)]$R$ is seminormal.	\end{enumerate}
	Then
		$\mathfrak{C}$ is free and in particular	$R$ is normal. 
\end{observation}

\begin{proof}
i) We proceed  by induction on $i$. The case $i=2$ follows from the following argument, and avoid its repetition. So, we may assume $i>2$ and inductively we assume  that the desired claim holds for $i-1$.	Let $\fp$ be in $\Spec(R)$ of height one. Let $A:={R_\fp}$. Recall
	that $(\mathfrak{C}_R)_\fp=\mathfrak{C}_A$. We apply $\mathfrak{C}^{(i-1)\otimes}\otimes_R-$ to $0\to \mathfrak{C}\to R\to R/ \mathfrak{C} \to 0$ and deduce
	the exact sequence $$0=\Tor_1^R(R, \mathfrak{C}^{(i-1)\otimes})\to \Tor_1^R(\frac{R}{\mathfrak{C}},\mathfrak{C}^{(i-1)\otimes} )\to \mathfrak{C}^{i\otimes}.$$
	Since $\Tor_1^R(\frac{R}{\mathfrak{C}},\mathfrak{C}^{(i-1)\otimes} )$ is both torsion and torsion-free it should be zero. Then $\Tor_1^A(\frac{A}{\mathfrak{C}_A},\mathfrak{C}_A^{(i-1)\otimes} )=0$. First assume $\mathfrak{C}\nsubseteq \fp$. This leads to the equality $\mathfrak{C}_\fp=A$ which is free as an $A$-module.
	So, we may assume that $\mathfrak{C}\subseteq \fp$. By dimension consideration,	$\mathfrak{C}_{R_\fp}$  primary to the maximal ideal, and is integrally closed, as $\dim({A})=1$ and 	$\mathfrak{C}_{A}$ is a common ideal of ${A}$ and $\overline{A}$ (see \cite[Ex. 12.2]{HS}). In view of \cite[3.3]{corso}, we know
	$\mathfrak{C}_A^{(i-1)\otimes}$ is free. Thanks to an inductive argument, $\mathfrak{C}_A$ is free. In sum, in both cases,  $\mathfrak{C}_A$ is free.
	This implies
	$A$ is normal. By dimension consideration,
	$A$ is regular. In other words, $R$ satisfies $(R_1)$.
Again, by Serre's criteria, $R$ is normal. This yields that $\mathfrak{C}=R$, which is free. 
	
The second and third cases are similar to the argument presented in  Observation \ref{ec}, and we leave the straightforward modifications to the reader.
\end{proof}
The following completes the proof of Theorem B) from introduction.
\begin{observation}
	If $R$ is one-dimensional. The following are equivalent:
\begin{enumerate}
\item[i)]$\Ext_R^i(\mathfrak{C},R)=0$ for some $i>0$,
\item[ii)] $R$ is Gorenstein,
\item[iii)]$\Ext_R^i(\overline{R},R)=0$ for all $i>0$,
\item[iv)]$\Ext_R^i(\overline{R},R)=0$ for some $i>0$.
\end{enumerate}
\end{observation}

\begin{proof}$i) \Rightarrow ii)$:
The case $i=1$ is in \cite[main theorem]{Re}. Without loss of generality assume that
$i>1$. After  completion if needed, we may assume $R$ possess a canonical module $\omega_R$. Set $(-)^+:=\Hom_R(-,\omega_R)$. Recall that $R$ is Cohen-Macaulay,
so $\Hom(\omega_R,\omega_R)=R$. We denote
$(-)^\vee:=\Hom_R(-,E_R(\frac{R}{\mathfrak{\fm}}))$, and  apply Matlis duality to see $\Ext^i_R(\mathfrak{C},\omega_R^+)^\vee=\Tor_{i-\dim R}
^R(\mathfrak{C},\omega_R)$  (see \cite[Lemma 3.5(2)]{ta}),
as dimension $R$ is one, and via a  shifting, we have $$\Tor_{i-1}
^R(\mathfrak{C},\omega_R)=\Tor_{i}
^R(\frac{R}{\mathfrak{C}},\omega_R)=0
\quad(\ast)$$There is nothing to prove if $\mathfrak{C}=R$, as it implies that 
$R$ is normal, and by a dimension-consideration regular, and in particular  Gorenstein.
Thus, without loss
of generality we may assume $\mathfrak{C}$ is a proper ideal. Thanks to \cite[Ex. 2.11]{HS} $\mathfrak{C}$ has a regular element. Consequently, $\mathfrak{C}$ is $\fm$-primary, and it is integrally closed, as $\dim({R})=1$ and 	$\mathfrak{C}$ is a common ideal of ${R}$ and $\overline{R}$ (see \cite[Ex. 12.2]{HS}). This along with $(\ast)$ implies that $\pd(\omega_R)<i$ (see \cite[3.3]{corso}). By Auslander-Buchsbaum formula, $\omega_R$ is free and consequently, $R$ is Gorenstein.

$ii) \Rightarrow iii)$ In view of  Auslander-Bridger formula $\Gdim_R(\overline{R})=0$, and the claim follows.

$iii) \Rightarrow iv)$ Trivial.

$iv) \Rightarrow i)$ Since $\overline{R}$ is principal ideal domain, and $0\neq\mathfrak{C} \lhd\overline{R}$, it should be principal, and so $\mathfrak{C}\cong\overline{R}$.
\end{proof}

The following extends Observation \ref{g}:

\begin{corollary}If  $\Ext_R^i(\mathfrak{C},R)=0$ for some $i>0$
	then $R$ satisfies $(G_1)$. In particular, $R$ is quasi-normal if
$R$ is $(S_2)$.
\end{corollary}

\section{$\mathfrak{C}$ and cohomological annihilator}
 By \textit{homological annihilator} we mean:

\begin{definition}(Vasconcelos).
	$\Hann_R(M):=\prod_{i=0}^d\Ann_R\Ext^i_R(M,R)$.
\end{definition}
If $R$ is Gorenstein, then  $\Ann(M)^{\Gdim (M)-\grade(M)+1}\subseteq\Hann(M)\subseteq\Ann(M),$ see e.g. \cite{a}. There are some of interests to find situations for which the equality holds.

\begin{observation}Let $R$ be  a 2-dimensional Gorenstein ring. Then $\Hann(\frac{\overline{R}}{R})=\Ann_R(\frac{\overline{R}}{R})$.
\end{observation}

\begin{proof} Recall that $\Ass(\Hom_R(\frac{\overline{R}}{R},R))=\Ass(R)\cap \Supp(\frac{\overline{R}}{R})=\Ass(R)\cap\V(\mathfrak{C}_R)=\emptyset$. From this, $\Hom_R(\frac{\overline{R}}{R},R)=0$. Also, $\overline{R}$ is maximal Cohen-Macaulay. It yields, $\Gdim_R(\overline{R})=0$. From  $0\to R\to  \overline{R}\to \frac{\overline{R}}{R} \to 0$ we know $\Gdim_R(\frac{\overline{R}}{R})\leq 1$. So, $\Ext^2_R(\frac{\overline{R}}{R},R) =0$. Also, it yields, $$0=\Hom_R(\frac{\overline{R}}{R},R)\to\Hom_R( {\overline{R}} ,R)\to\Hom_R( {{R}} ,R)\to \Ext^1_R(\frac{\overline{R}}{R},R)\to \Ext^1_R({\overline{R}},R)=0.$$This gives
$\Ext^1_R(\frac{\overline{R}}{R},R)=\frac{{R}}{\mathfrak{C}}$. We deduce that 	$$\Hann(\frac{\overline{R}}{R}):=\prod_{i=0}^2\Ann_R\Ext^i_R(\frac{\overline{R}}{R},R)=\Ann_R\frac{{R}}{\mathfrak{C}}=\mathfrak{C}=\Ann_R(\frac{\overline{R}}{R}),$$as claimed.
\end{proof}

\begin{definition}(Bridger-Roberts). By \textit{cohomological annihilator} we mean
	$\Ann^h_R(M):=\prod_{i=0}^{\dim R-1}\Ann_RH^i_{\fm}(M)$.
\end{definition}

\begin{observation}\label{ha}Let $R$ be  a  Gorenstein ring. Then
	\begin{enumerate}
	\item[i)]  $\mathfrak{C}_R\subseteq \Ann^h(\overline{R})$, 	\item[ii)] Suppose the ring is not normal. Then $\rad \mathfrak{C}_R \neq \rad \Ann^h(\overline{R})$.
\end{enumerate}	
\end{observation}
An ideal is called semiconductor if it defines the same locus as $\mathfrak{C}_R$.
\begin{proof}
i) 
Without loss of generality, we may assume the ring is complete. Recall by local duality that $H^{d-j}_{\fm}(\overline{R})^\vee=\Ext^{j}_R({\overline{R}},R).$ This gives $\Ann(H^{d-j}_{\fm}(\overline{R}))=\Ann(\Ext^{j}_R({\overline{R}},R))$.
 From  $0\to R\to  \overline{R}\to \frac{\overline{R}}{R} \to 0$ we know for any $j$ that
 \begin{enumerate}
 	\item[$(\ast)_{j=1}$]:   $\Ext^1_R(\frac{\overline{R}}{R},R)\to\Ext^1_R({\overline{R}},R)\to\Ext^1_R({{R}},R) =0$, 	\item[$(+)_{j>1}$]: $0=\Ext^{j-1}_R({R},R)\to\Ext^j_R(\frac{\overline{R}}{R},R)\to\Ext^{j}_R({\overline{R}},R)\to\Ext^1_R({{R}},R) =0$.
 \end{enumerate}	
 From $(\ast)_{j=1}$ we have $\mathfrak{C}_R\subseteq \Ann(\Ext^{1}_R(\frac{\overline{R}}{R},R))\subseteq \Ann(\Ext^{1}_R({\overline{R}},R))$. 
From $(+)_{j>1}$ we have $\mathfrak{C}_R\subseteq \Ann(\Ext^{J}_R(\frac{\overline{R}}{R},R))= \Ann(\Ext^{1}_R({\overline{R}},R))$. This gives 
 $\mathfrak{C}_R\subseteq \Ann^h(\overline{R})$.

ii) Suppose $\dim R=1$. Since $\mathfrak{C}_R$ is nonzero, proper, it is $\fm$-primary,
 and so $\rad \mathfrak{C}_R=\fm$. But, $\rad \Ann^h(\overline{R})=\rad \Ann(H^0_\fm(\overline{R}))=R$. Now, suppose  $\dim R>1$ and suppose on the way of contradiction that $$\rad (\mathfrak{C}_R)   =\rad(\Ann^h(\overline{R}))\quad(\dagger)$$
 \begin{enumerate}
	\item[$Fcat$:] (Bridger, see \cite{brid}).
Set
$\gamma(M):=\bigcap _{i>0}\rad(\Ann_R\Ext^i_R(M,R))$ and suppose $\Gdim(M)<\infty$. Then $\gamma(M)$ contains an $M$-sequence of length $k$ iff $M$ is $k$-torsion-free.
 \end{enumerate}
Now, since $\overline{R}$ is $(S_2)$, it is reflexive and so 2-torsionless. Thanks to the above fact,
$\gamma(\overline{R}) $ contains a regular $\overline{R}$-sequence of length two.  In view of $(\dagger)$ we deduce $\mathfrak{C}_R$ contains a regular $\overline{R}$-sequence of length two. In particular, its height is bigger then one. But, the
$(S_2)$ condition of $R$ implies $\Ht(\mathfrak{C}) =1$,  a contradiction.
\end{proof}
As may the reader see, part $i)$ works for Cohen-Macaulay rings with canonical module. It may be nice to drop the Cohen-Macaulay assumption:
\begin{corollary}\label{recma}
	Let $R$ be so that there is a Cohen-Macaulay ring $G$ between $R\subseteq \overline{R}$.  Then $\mathfrak{C}_R\subseteq \Ann^h(\overline{R})$.
\end{corollary}

\begin{proof}We may assume the ring is complete.
 Recall that
	$$\mathfrak{C}_R:=\Ann_R(\frac{\overline{R}}{R})\subseteq\Ann_R(\frac{\overline{R}}{G})\subseteq\Ann_G(\frac{\overline{R}}{G})=\Ann_G(\frac{\overline{G}}{G})=\mathfrak{C}_G\quad(\ast)$$Also,
	by the independence theorem for local cohomology we have
	$$\Ann^h_G(\overline{G})\cap R=\Ann^h_R(\overline{R})\quad (+)$$

	 From $(\ast)$ and Observation \ref{ha} we have $\mathfrak{C}_R\subseteq \mathfrak{C}_G\subseteq \Ann^h_G(\overline{G}). $  Thus, combining with $(+)$ we deduce $\mathfrak{C}_R\subseteq \Ann^h_G(\overline{G})\cap R= \Ann^h_R(\overline{R})$ as claimed.
\end{proof}

The next result is trivial if $i>1$ or $\overline{R}$ is Cohen-Macaulay:
\begin{corollary}
	Let $R$ be  a Cohen-Macaulay ring  and $\underline{x}$ be a  full parameter sequence. Then $\mathfrak{C}_R\subseteq \Ann(\Tor^R_i(\overline{R},R/\underline{x}R))=0$ for all $i>0$. 	
\end{corollary}

	\begin{proof}We may assume the ring is complete.
		Let $F_\bullet$ be the Koszul complex with respect to $\underline{x}$, and set $M:=\overline{R}$. Then $F_\bullet$ is a complex of finite free  modules with homology
		of finite length. By \cite[Ex. 8.1.6]{BH} $\Ann^h(M)$ annihilates $H^+(F_\bullet \otimes_RM)$. Since
		$\mathfrak{C}_R\subset\Ann^h(\overline{R})$ we get the desired claim. 
	\end{proof}
Let k be a field and $\underline{m} = (m_1, \ldots,m_n)\in\mathbb{N}_0^n$ with $m_1 + \ldots + m_n = d$, we denote by $P_{n,d,\underline{m}}$ the pinched
Veronese ring in $n$ variables formed by removing the monomial generator $x_1^{m_1}\dots 
x_n^{m_n}$ from the
d-Veronese subring of $k[[x_1, \ldots , x_n]].$  The rings $P_{n,d,\underline{m}}$ are rarely Cohen-Macaulay. Despite this:

\begin{example}Let $R:=P_{n,d,\underline{m}}$. Then
$\mathfrak{C}_{R}\subseteq \Ann^h(\overline{R})$.
\end{example}

\begin{proof}
Recall from \cite[Theorem 2.3]{pinc} that $\overline{P_{n,d,\underline{m}}}$ is the  Veronese ring with the exception of normal rings $P_{n,d,d}$ and $P_{2,2,(1,1)}$. By \cite[Theorem 1.1]{pinc} we observe that $\overline{R}$ is Cohen-Macaulay. So, the desired claim follows by Corollary \ref{recma}.
\end{proof}
Suppose $R$ is  2-dimensional and Cohen-Macaulay.
Since $\depth(\Hom_R({\overline{R}},R))\geq\min\{2,\depth({R})\}$, it follows that $\mathfrak{C}_R$ is maximal Cohen-Macaulay. This extends to higher dimensional as follows:

\begin{observation}
If $R$ is Gorenstein and $\overline{R}$ is Cohen-Macaulay, then $\mathfrak{C}_R$ is maximal Cohen-Macaulay.
\end{observation}

\begin{proof}
	Without loss of generality we assume that $d:=\dim R>1$ and we assume $R$ is not normal.  We observe that $\Ext^1_R(\frac{\overline{R}}{R},R)=\frac{{R}}{\mathfrak{C}}$. From  $0\to R\to  \overline{R}\to \frac{\overline{R}}{R} \to 0$ we know $$d-1=\dim (\frac{\overline{R}}{R})\geq \depth_R(\frac{\overline{R}}{R})\geq d-1.$$
	Thus, $\frac{\overline{R}}{R}$ is Cohen-Macaulay. This gives
	Cohen-Macaulayness of $\Ext^1_R(\frac{\overline{R}}{R},R)\cong\frac{{R}}{\mathfrak{C}}$ which is of dimension $d-1$. In view of 
$0\to \mathfrak{C}\to  \overline{R}\to \frac{\overline{R}}{\mathfrak{C}} \to 0$ we observe $\mathfrak{C}$ is maximal Cohen-Macaulay.\end{proof}

\begin{example}\label{3.9}
i)	Let $G$ be the set of all quadric monomials in $\{X,Y,Z\}$. Let $R:=k[[G\setminus\{XY\}]]$. Since $R$ is Gorenstein (see \cite[Theorem A]{pinc})
	and $\overline{R}=k[[G]]$ is  Cohen-Macaulay, we see
$\mathfrak{C}$ is maximal Cohen-Macaulay.
	
	ii) The Gorenstein assumption of $R$ is needed. Indeed, 
Let $R:=k[[s^2,s^3,st,t]]$. Here, we claim that $\mathfrak{C}=\fm$:
First, we compute it.
Since $s=s^3/s^2$ it belong to fraction field of $R$. Also, it is integral over $R$, as $f(s)=0$ where $f[X]:=X^2-s^2$. From this, $\overline{R}=k[[s,t]]$. Now, we compute its conductor: 
\begin{enumerate}
	\item[a)] In order to show
	$t\in \mathfrak{C}$ we note that $t.\overline{R}\subseteq R$. Indeed,  $\{t.1,t.s,t.t\}\subset R .$	\item[b)] In order to show
	$s^2\in \mathfrak{C}$ we note that $s^2.\overline{R}\subseteq R$. Indeed, $\{s^2.1,s^2.s,s^2.t\}\subset R. $	\item[c)]  In order to show
	$s^3\in \mathfrak{C}$ we note that $s^3.\overline{R}\subseteq R$. Indeed, $\{s^3.1,s^3.s,s^3.t\}\subset R. $
	\item[d)] In order to show
	$st\in \mathfrak{C}$ we note that $st.\overline{R}\subseteq R$. Indeed,  $\{st.1,st.s,st.t\}\subset R. $ 
\end{enumerate}

Thus, $\mathfrak{C}=\fm_R$. Since  $\overline{R}$ is generated as an $R$-module by  $\{1,t\}$ and that $\fm t\subseteq R$, we deduce that $\frac{\overline{R}}{R}=\frac{{R}}{\fm}$.
Since $\depth(\fm)=1$, we deduce that	$\mathfrak{C}_{R}$ is not maximal  Cohen-Macaulay.

%iii) The Cohen-Macaulay assumption of $\overline{R}$ is needed. 
\end{example}
The following completes the proof of Theorem C) from introduction.
\begin{observation}\label{bb}Let $R$ be  a 2-dimensional analytically unramified ring satisfying $(R_1)$. Then $\mathfrak{C}_R=\Ann(H^1_\fm(R))$.
\end{observation}

\begin{proof}
	 From  $0\to R\to  \overline{R}\to \frac{\overline{R}}{R} \to 0$ we deduce $0=H^0_\fm(\overline{R})\to H^0_\fm( \frac{\overline{R}}{R})\to H^1_\fm(R)\to H^1_\fm(\overline{R})=0.$ According to the $(R_1)$-condition, we see the ring is normal over the punctured spectrum. From this $H^0_\fm( \frac{\overline{R}}{R})=\frac{\overline{R}}{R}$.
	 Consequently, $H^1_\fm(R)=H^0_\fm( \frac{\overline{R}}{R})=\frac{\overline{R}}{R}$, so $\mathfrak{C}_R=\Ann(H^1_\fm(R))$ as claimed.
\end{proof}

Let $M$ be an $R$-module. Vasconselos proved that $M$ is projective provided
$M\otimes_R\overline{R}$ is $\overline{R}$-projective and that $M\otimes_RR/\mathfrak{C}$ is $R/\mathfrak{C}$-projective. Two natural questions arises:
\begin{enumerate}
	\item[$1)$]:  Suppose $\Gdim_{\overline{R}}(M\otimes_R\overline{R})=0$ and that $\Gdim_{R/ \mathfrak{C}}(M\otimes_R\frac{R}{ \mathfrak{C}})=0$.
Is  $\Gdim_R(M)=0$?
		\item[$2)$]: Suppose $\pd_{\overline{R}}(M\otimes_R\overline{R})<\infty$  and that $\pd_{R/ \mathfrak{C}}(M\otimes_R\frac{R}{ \mathfrak{C}})<\infty$.
	Is 	$\pd_R(M)<\infty$?
\end{enumerate}
We show non of these true:
\begin{example}\label{3.9b}
 \begin{enumerate} 
	\item[$1)$]: Let $R:=k[[s^3,s^4,s^5]]$ and $M:=\fm=\mathfrak{C}$. Then $\Gdim_R(M)\neq0$ but
$\Gdim_{\overline{R}}(M\otimes_R\overline{R})=\Gdim_{R/ \mathfrak{C}}(M\otimes_R\frac{R}{ \mathfrak{C}})=0$.
\item[$2)$]:	Let $R:=k[[s^2,s^3,st,t]]$ and $M:=k$. Then: $\pd_R(M)=\infty$ but
$\pd_{\overline{R}}(M\otimes_R\overline{R})<\infty$  and that $\pd_{R/ \mathfrak{C}}(M\otimes_R\frac{R}{ \mathfrak{C}})<\infty$.
\end{enumerate} 
\end{example}
\section{Ikeda's problem}
	
\begin{conjecture}\label{41}
Suppose $R$ is analytically unramified. Then $\mathfrak{C}\nsubseteq \underline{x}R$ where $\underline{x}$ is a regular sequence.
\end{conjecture}

We only use the following weakened property of quasi-Gorenstein:

\begin{lemma}
	Let $(R,\fm)$ be d-dimensional. If $\id_R(H^d_\fm(R))<\infty,$ then $H^d_\fm(R)$ is injective.
\end{lemma}
\begin{proof}
Easy.
\end{proof}

The following drop some conditions  from \cite{i}, for example the Cohen-Macaulay condition:
\begin{lemma}\label{injh}
Suppose $R$ is analytically unramified and quasi-Gorenstein. Then $\Hom_R(\mathfrak{C},H^d_\fm(R))=H^d_\fm(\overline{R})$, where $d:=\dim R$.
\end{lemma}

\begin{proof}
Since $\overline{R}$ is finitely generated and $H^d_\fm(R)$ is injective, we have the following functorial isomorphisms:$$
\begin{array}{ll}
\Hom_R(\mathfrak{C},H^d_\fm(R))&=\Hom_R(\Hom_R(\overline{R},R),H^d_\fm(R))\\&\cong \overline{R}\otimes\Hom_R(R,H^d_\fm(R))\\&=\overline{R}\otimes H^d_\fm({R})
\\&\stackrel{(\ast)}\cong H^d_\fm(\overline{R}),
\end{array}$$where $(\ast)$ follows because $d$ is the cohomological dimension.
\end{proof}
The following extends and simplifies a result of Ikeda:
\begin{observation}\label{qc}
 Suppose $R$ is quasi-Gorenstein. Then $\mathfrak{C}\nsubseteq \underline{x}R$ where $\underline{x}$  is a system of parameter. 
\end{observation}

\begin{proof}
Suppose $\mathfrak{C}\subseteq \underline{x}R$ for some     system of parameter $\underline{x}$. We  extend it to a sequence of length $d$. By canonical element conjecture  which is a theorem, see \cite[Section 9.3]{BH}, there is a natural  map $\rho_{\overline{R}}:\frac{\overline{R}}{\underline{x}\overline{R}}\to H^d_\fm(\overline{R})$ with the property that $\alpha:=\rho_{\overline{R}}(1+\underline{x}\overline{R})\neq 0 $. Now, look at $\alpha\in H^d_\fm(\overline{R}) \stackrel{\ref{injh}}=\Hom_R(\mathfrak{C},H^d_\fm(R))$.
In particular, there is a nonzero $f:\mathfrak{C}\to H^d_\fm(R)$ so that $\alpha=\phi(f)$ where $\phi:H^d_\fm(\overline{R})\stackrel{\cong}\lo\Hom_R(\mathfrak{C},H^d_\fm(R))$. Also, the assignment $m\otimes (M\stackrel{g}\lo N)\mapsto g(m)$ gives the evaluation map
$\varPhi_{M,N}:M\otimes\Hom(M,N)\lo N$. Here, we set
\begin{enumerate}
	\item[i)]   $\varPhi_1:=\varPhi_{\mathfrak{C},H^d_\fm(R)}:\mathfrak{C}\otimes\Hom(\mathfrak{C},H^d_\fm(R))\lo H^d_\fm(R)$, 	\item[ii)] $\varPhi_2=\varPhi_{\overline{R}, R}\otimes_R\id_{\frac{{R}}{\underline{x}{R}}}:[\Hom(\overline{R},R)\otimes\overline{R}\lo R]\otimes_R\frac{{R}}{\underline{x}{R}}$.
\end{enumerate}	
We have the following commutative diagram$$\xymatrix{
	& \mathfrak{C}\otimes_RH^d_\fm(\overline{R})\ar[r]^{\varPhi_1} &H^d_\fm(R)	\\   & \mathfrak{C}\otimes_R \frac{\overline{R}}{\underline{x}\overline{R}}\ar[u]^{\id_\mathfrak{C}\otimes \rho_{\overline{R}}}\ar[r]^{\varPhi_2}  &  \frac{{R}}{\underline{x}{R}}\ar[u]_{\rho_{{R}}}
	&&&}$$Recall that $f:\mathfrak{C}\to H^d_\fm(R)$ so that $\alpha=\phi(f)$  is nonzero. Take $y\in \mathfrak{C}$ so that $f(y)\neq 0$. Since $\mathfrak{C}\overline{R}\subseteq R$ we deduce that $ry\in R$ and so  $\varPhi_2(y\otimes(r+\underline{x}\overline{R}))=ry+\underline{x}R\in\frac{{R}}{\underline{x}{R}}  $. Now, recall that $y\in \mathfrak{C}\subseteq \underline{x}R$. This gives
$$ \rho_{ R } \circ\varPhi_2  ((y\otimes(1+\underline{x}\overline{R}))=\rho_R(y+\underline{x} R)=0.$$
But,
$\varPhi_1\circ(\id_\mathfrak{C}\otimes \rho_{\overline{R}})((y\otimes(1+\underline{x}\overline{R}))=\varPhi_1(y\otimes\alpha)=\varPhi_1(y\otimes\phi(f))=f(y)\neq0.$\end{proof}
The following removes dimension restriction from \cite[12.4]{HS} and \cite{i}.

\begin{observation}
 Suppose $R$ is analytically unramified. Then  $\mathfrak{C}\nsubseteq {x}R$ where ${x}\in \fm$.  
\end{observation}
\begin{proof}
	Suppose  $\mathfrak{C}\subseteq {x}R$ for some ${x}\in \fm$.  Since $R$ is analytically unramified, it is reduced. It is easy to see $\overline{R}/R$ is analytically unramified, and so reduced. Due to \cite[Ex. 2.11]{HS} we know $\mathfrak{C}$ has a regular element.
	Thus, $\grade({x}R,R)=1$. Consequently, $x$ is regular. Suppose $\mathfrak{C}=(a_1,\ldots,a_n)$ and let $b_i$ so that $a_i=b_ix$. Let $J:=(b_1,\ldots,b_n)$. Then $\mathfrak{C}=xJ$. Thus, $J\overline{R}=(1/x)\mathfrak{C}\overline{R}=(1/x)\mathfrak{C}=J$. In other words, $J$ is a common ideal of $\{R,\overline{R}\}$. So, $J\subseteq \mathfrak{C}$. Therefore, $J=xJ$. By Nakayama's lemma $J=0$. This is a contradiction, since $\mathfrak{C}\neq 0$.  
\end{proof}
\begin{corollary}
	Suppose $\depth(R)=1$. Then  $\mathfrak{C}\nsubseteq \underline{x}R$ where $\underline{x}$  is a regular sequence. 
\end{corollary}

\begin{example}
	Let $R:=k[[s^2,s^3,st,t]]$. Then  $\mathfrak{C}\nsubseteq \underline{x}R$ where $\underline{x}$  is a parameter sequence.
\end{example}
\begin{proof}
Thanks to Example \ref{3.9} we know $(t,s^2,s^3,st)=\mathfrak{C}$, and the equality follows as the ring is not normal. Suppose on the way of contradiction that 	 $\mathfrak{C}\subseteq \underline{x}R$.
	Then $\fm\subseteq \underline{x}R$. This implies the ring is regular, a contradiction.
\end{proof}
The previous example extends as follows:
\begin{observation}
Suppose $R$ is a two-dimensional Buchsbaum ring which is $(R_1)$. Then $\mathfrak{C}\nsubseteq \underline{x}R$ where $\underline{x}$  is a system of parameter.
\end{observation}
\begin{proof}
Recall from Observation \ref{bb} that $\frac{\overline{R}}{R}=H^1_\fm(R)=\oplus_X k$ where $X$ is an index set. Without loss of generality we may assume that $R$ is not $(S_2)$, as otherwise  $\mathfrak{C}=R$ which is not   subset of $ \underline{x}R$. This implies $X$ is non-empty.
Suppose on the way of contradiction that 	 $\mathfrak{C}\subseteq \underline{x}R$.
Then $\fm\subseteq \underline{x}R$. This implies the ring is regular, and so normal. Thus $R=\mathfrak{C}\subseteq \underline{x}R$, a contradiction.
\end{proof}
\begin{corollary}
	Suppose $R$ is a complete two-dimensional seminormal ring which is $(R_1)$. Then $R$ is  Buchsbaum.
\end{corollary}
\begin{proof}We may and do assume that the ring is not $(S_2)$, as otherwise it is Cohen-Macaulay and so  Buchsbaum.
	Recall from  $(R_1)$ condition that $\mathfrak{C}$ is $\fm$-primary. By seminormality, we know $\mathfrak{C}$ is a radical ideal in $\overline{R}$. Thus,  $\mathfrak{C}=\fm_{\overline{R}}$. Then $\mathfrak{C}=\mathfrak{C}\cap R=\fm_{\overline{R}}\cap R=\fm$. As $\mathfrak{C}=\Ann (H^1_\fm(R))$, it turns out that $R$ is  Buchsbaum.
\end{proof}

\begin{example}
	Let $R:=k[[X^2,Y^2,XY,X^2Y,XY^2]]$. Then
	
		\begin{enumerate}
		\item[i)] $R$ is seminomal,	\item[ii)] $\mathfrak{C}\nsubseteq \underline{x}R$ where $\underline{x}$  is a parameter sequence.	%\item[iii)]  $R$ is not $(R_1)$.
	\end{enumerate}

\end{example}
\begin{proof}
i) This is well-known.

ii) It is easy to see $\overline{R}=k[[X,Y]]$. Let $J:=(XY,X^2Y,XY^2)$. In order to show
	$J\subseteq \mathfrak{C}$ we note by the routine argument of Example \ref{3.9}(ii) that $J\overline{R}\subseteq R$. Suppose on the way of contradiction that 	 $\mathfrak{C}\subseteq(a,b )R$ where $a,b$ is a parameter sequence.
	Then $\fm\subseteq\langle J,X^2,Y^2\rangle\subseteq \langle a,b,X^2,Y^2\rangle\subseteq \fm$. From this we see $\langle a,b,X^2,Y^2\rangle= \fm$. But, $5=\mu(\fm)\leq4$, a contradiction.
\end{proof}

We close this section by presenting some higher dimensional examples.

\begin{example}
	Let $G$ be the set of all quadric monomials in $\{X,Y,Z,W\}$. Let $R:=k[[G\setminus\{ZW\}]]$. Then $\mathfrak{C}\nsubseteq \underline{x}R$ where $\underline{x}$  is a parameter sequence.
\end{example}

\begin{proof}
	It is easy to see $\overline{R}=k[[G]]$, the 2-Veronese subring of $k[[X,Y,Z,W]]$. Let $J:=\langle G\setminus\{ZW,Z^2,W^2\}\rangle_R$. In order to show
	$J\subseteq \mathfrak{C}$, we note by the routine argument presented in Example \ref{3.9}(ii), that $J\overline{R}\subseteq R$. Suppose on the way of contradiction that 	 $\mathfrak{C}\subseteq(a,b,c,d )R$ where $a,b,c,d$ is a parameter sequence.
	Then $$\fm\subseteq\langle J,Z^2,W^2\rangle\subseteq \langle a,b,c,d,Z^2,W^2\rangle\subseteq \fm.$$ From this we see  $9=\mu(\fm)\leq 6$, a contradiction.
\end{proof}

The following completes the proof of Theorem A) from introduction.

\begin{observation}\label{oneless}Let $R$ be analytically unramified containing a field and suppose
$\fm=\mathfrak{C}+y R$ for some $y$. Then $\mathfrak{C}\nsubseteq \underline{x}R$ where $\underline{x}$  is a system of parameter.
\end{observation}
\begin{proof}Without loss of generality we assume the ring is complete. In particular,
$R$ a reduced ring. Suppose on the way of contradiction that there is  some  system of parameter $\underline{x}:=x_1,\ldots,x_d$ so that $\mathfrak{C}\subseteq \underline{x}R$. Then $\fm=(x_1,\ldots,x_d,y)$. Let $S:=k[[X_1,\ldots,X_d,X]]$ and consider the natural projection $\pi:S\twoheadrightarrow R$. Since $R$ is reduced we see $\ker(\pi)$ is radical and of height one. Since $S$ is regular, $\ker(\pi)=(f)$ is principal. From this $R$ is hypersurface, in particular quasi-Gorenstein. This is in contradiction with 
Observation \ref{qc}.
\end{proof}

Here, we present some non-trivial examples for which $\fm=\mathfrak{C}+x R$:

\begin{example}
	Let $G$ be the set of all cubic monomials in $\{X,Y,Z\}$. Let $R:=k[[G\setminus\{X^2Y\}]]$. Then $\fm=\mathfrak{C}+X^3 R$. In particular, $\mathfrak{C}\nsubseteq \underline{x}R$ where $\underline{x}$  is a parameter sequence.	%\item[iii)]  $R$ is not $(R_1)$.
\end{example}

\begin{proof}
As $X^2Y=\frac{(XY^2)XYZ}{Y^2Z}$ we see $X^2Y\in Q(R)$.
Let $f(t):=t^2-(X^3)(XY^2)\in R[t]$. Since $f(X^2Y)=0$, we see $\overline{R}=k[[G]]$. Let $J:=\langle G\setminus\{X^2Y,X^3\}\rangle_R$. In order to show
	$J\subseteq \mathfrak{C}$ we show
	$\overline{R}J\subseteq R$. To this end, we show $X^2YJ\subseteq R$. Let us check this by looking at each seven generators of $J$:
\begin{enumerate}
\item[1)] $(X^2Y) Z^2Y=XY^2(Z^2X) \in J \subseteq R, \quad\quad\quad	 2)\ \  (X^2Y) Y^3=(XY^2)Y^3   \in J \subseteq R, $  	\item[3)] 	$(X^2Y) ZXY=(Y^2X) ZX^2\in J \subseteq R, \quad\quad\ \ 4)  \ \	 (X^2Y) Z^3=(ZX^2)Z^2Y \in J \subseteq R,$   	\item[5)] 	$(X^2Y)ZY^2=(Y^3) ZX^2 \in J \subseteq R,\quad\quad\ \ \ \ \ \ \   6)  \ \ 	 (X^2Y) Z^2X=(X^3)Z^2Y  \in J  \subseteq R,$  
\item[7)]  	$(X^2Y)  ZX^2=(X^3) XYZ \in J  \subseteq R$.
	\end{enumerate}
Thus, $\fm\subseteq J+X^3R\subseteq \mathfrak{C}+X^3R\subseteq \fm$.	This proves the first part. The particular case is in Observation \ref{oneless}.
\end{proof}

\begin{example}
	Let $k$ be a field, $n>1$ and $\{X_1,\ldots,X_n\}$ be indeterminate. Let $$R:=k[[X_1,\ldots,X_{n-1},X_{n}(X_1,\ldots,X_n)]]\subseteq k[[X_1,\ldots,X_n]].$$ Then $\fm=\mathfrak{C}+X_{n}^2 R$. In particular, $\mathfrak{C}\nsubseteq \underline{x}R$ where $\underline{x}$  is a parameter sequence.	%\item[iii)]  $R$ is not $(R_1)$.
\end{example}

\begin{proof}
As $X_n=\frac{X_nX_1}{X_1}$ we see $X_n\in Q(R)$.
Let $f(t):=t^2-X_n^2\in R[t]$. Since $f(X_n^2)=0$
we can deduce that  $\overline{R}=k[[X_1,\ldots,X_n]]$. Let $J:=\langle X_1,\ldots,X_{n-1},X_1X_n,\ldots,X_{n-1}X_n\rangle_R$. In order to show
$J\subseteq \mathfrak{C}$ we show
$\overline{R}J\subseteq R$. To this end, we show $X_nJ\subseteq R$. Let us check this by looking at each $2n-2$ generators of $J$:
\begin{enumerate}
\item[$\bullet_j$] $(X_n) X_j\in J \subseteq R$ for all $1\leq j<n$	\item[$\bullet_i$] $(X_n) X_iX_n=X_iX_n^2\in J \subseteq R$ for all $1\leq i<n$.
		
	\end{enumerate}
	Thus, $\fm\subseteq J+X_n^2R\subseteq \mathfrak{C}+X_n^2R\subseteq \fm$.	This proves the first part. The particular case is in Observation \ref{oneless}.
\end{proof}

%Next, consider the following motivational example:

\end{document}